\documentclass{article}
\usepackage{amsmath,amsfonts,amssymb,amsthm}
\usepackage{epsfig}
\usepackage{mathrsfs}
\usepackage{enumerate}
\usepackage{color}
\newtheorem{theorem}{Theorem}

\allowdisplaybreaks
\theoremstyle{remark}
\newtheorem{rem}{Remark}[section]

\newcommand{\Nor}{{\mathrm{Nor}}\,}
\newcommand{\pos}{{\mathrm{pos}}\,}

\newcommand{\R}{{\mathbb R}}

\catcode`@=11
\def\section{%
\setcounter{equation}{0} \setcounter{theorem}{0} \@startsection
{section}{1}{\z@}{-4.0ex plus -1ex minus
    -.2ex}{2.3ex plus .2ex}{\bf}}
\catcode`@=12 \theoremstyle{definition}

\begin{document}

\title{Kinematic formulas for area measures}

\author{Paul Goodey, Daniel Hug, and 
Wolfgang Weil}
\date{\today}

\maketitle
\begin{abstract}
We obtain a Principal Kinematic Formula and a Crofton Formula for surface area measures of convex bodies, both involving linear operators on the vector space of signed measures on the unit sphere $S^{d-1}$. These formulas are related to a localization of Hadwiger's Integral Geometric Theorem. The operators, mentioned above, will be shown to be  compositions of spherical Fourier transforms originating in the work of Koldobsky. As an application of our Crofton Formula, we will find an extension of Koldobsky's orthogonality relation for such transforms from the case of even spherical functions to centered functions.

\medskip\noindent
{\it Key words:} Principal kinematic formula, Crofton formula, surface area measures, convex bodies, spherical Fourier transforms, orthogonality relation.

\medskip\noindent
{\it 2010 Mathematics Subject Classification:} 52A20, 52A22, 52A39, 53C65.
\end{abstract}

\section{Introduction}

Let ${\cal K}$ be the space of compact convex sets in $\R^d, d\ge 3,$ supplied with the Hausdorff metric, and let ${\cal K}'$ be the subset of non-empty elements (convex bodies). For $K\in {\cal K}$ and $j\in\{0,....,d\}$, let $V_j(K)$ denote the $j$th intrinsic volume of $K$ and $C_j(K,\cdot)$ the $j$th curvature measure. For $j\in\{1,....,d-1\}$,  $S_{j}(K,\cdot )$ is the $j$th area measure of $K$ (see \cite{S}, for the standard notions from convex geometry used in this article). 
Let $G(d,j)$ and $A(d,j)$ be the manifolds of $j$-dimensional linear (respectively, affine) subspaces of $\R^d$, supplied with their natural topologies and (suitably normalized) invariant measures $\nu_j$ and $\mu_j$. Moreover, $G_d$   is the group of rigid motions with Haar  measure $\mu$ and $SO_d$ is the rotation group with Haar probability measure $\nu$ (see \cite{SW}).

Two basic results in integral geometry are the {\it Principal Kinematic Formula (PKF)} 
\begin{equation}\label{PKF}
\int_{G_d} V_j(K\cap gM)\,\mu(d g) = \sum_{k=j}^d c(d,j,k)V_k(K) V_{d+j-k }(M), 
\end{equation}
which holds for $j= 0,\dots ,d,$ 
and the {\it Crofton Formula (CF)}
\begin{equation}\label{CF} 
\int_{A(d,q)} V_j(K\cap E)\,\mu_q(d E) = c(d,j,q) V_{d+j-q}(K), 
\end{equation}
which holds for $j=0,\dots,q$. 
Here $0\le q\le d-1$ (the case $q=d$ is trivial), $K,M\in{\cal K}$ and the $c(d,j,k)$ are known constants.

For the curvature measures, there exist local versions of PKF and CF, namely
\begin{equation}\label{PKF2}
\int_{G_d} C_j(K\cap gM, A\cap gB)\,\mu(d g) = \sum_{k=j}^d \tilde c(d,j,k)C_k(K,A) C_{d+j-k }(M,B), 
\end{equation}
$j= 0,\dots ,d,$ and
\begin{equation}\label{CF2} 
\int_{A(d,q)} C_j(K\cap E, A\cap E)\,\mu_q(d E) = \tilde c(d,j,q) C_{d+j-q}(K,A), 
\end{equation}
$j=0,\dots,q,$
which hold for arbitrary Borel sets $A,B\subset\R^d$ and known constants $\tilde c(d,j,q)$.

It is easy to see that corresponding results for area measures cannot hold in exactly the same form. For example, a CF for area measures would state that
\begin{equation}\label{CF3} 
\int_{A(d,q)} S_j(K\cap E, A\cap L(E))\,\mu_q(d E) = \bar c(d,j,q) S_{d+j-q}(K,A), 
\end{equation}
for $j=1,\dots,q-1$, some constants $\bar c(d,j,q)$ and Borel sets $A$ in the unit sphere $S^{d-1}$ (it is obvious that the intersection of $A$ has to be taken not with the affine subspace $E$, but with the linear subspace $L(E)$ parallel to $E$). Formula \eqref{CF3} is not true for polytopes $K$. Namely, for a polytope $K$, the measure $S_j(K,\cdot)$ is concentrated on the $(d-j-1)$-dimensional boundary parts of a spherical cell complex. More precisely, these boundary parts arise as the spherical images of the $j$-dimensional faces of $K$ and, on each such spherical image, $S_j(K,\cdot)$ is a positive multiple of the corresponding Hausdorff measure. If we choose the Borel set $A$ in \eqref{CF3} to be the support of $S_{d+j-q}(K,\cdot)$,  then, for $\mu_q$-almost all $E$, the intersection $A\cap L(E)$ has dimension $q + d-(d+j-q)-1 -d = 2q-j-1-d< d-j-1$ and therefore $S_j(K\cap E, A\cap L(E)) =0$. Thus the left side of \eqref{CF3} vanishes. 

As a variant, one could replace the integrand by $S'_j(K\cap E, A\cap L(E))$, where the prime indicates that the area measure is calculated in the subspace $L(E)$. Then, for polytopes, the support of $S'_j(K\cap E, \cdot)$ has dimension $q-j-1$ and, since $2q-j-1-d < q-j-1$, the integral in \eqref{CF3} still vanishes.

A generalization of \eqref{PKF2} and \eqref{CF2}, which involves area measures, was investigated by Glasauer \cite{Gl}, who showed corresponding formulas for the support measures $\Theta_i(K,\cdot)$, namely
\begin{equation}\label{PKF4}
\int_{G_d} \Theta_j(K\cap gM, A\wedge gB)\,\mu(d g) = \sum_{k=j}^d \tilde c(d,j,k)\Theta_k(K,A) \Theta_{d+j-k }(M,B), 
\end{equation}
for $j= 0,\dots ,d$, and
\begin{equation}\label{CF4} 
\int_{A(d,q)} \Theta_j(K\cap E, A\wedge E)\,\mu_q(d E) = \tilde c(d,j,q) \Theta_{d+j-q}(K,A), 
\end{equation}
for $j=1,\dots,q-1$, with the same constants $\tilde c(d,j,k)$ as above. The support measures are common generalizations of the curvature and the area measures and are supported on the normal bundle of the bodies. To be more precise, the normal bundle $\Nor K$ of a body $K$ is defined as
$$
\Nor K = \{ (x,u) : x {\rm\ a\ boundary\ point\ of\ } K, u {\rm\ an\ outer\ normal\ of\ }K{\rm\ at\ }x\}
$$
and the formulas above hold for Borel sets $A\subset\Nor K, B\subset\Nor M$ with
\begin{align*}
A\wedge B = \{ (x,u) : &{\rm\ there\ are\ }u_1,u_2\in S^{d-1}{\rm\ such\ that\ }\cr
&(x,u_1)\in A, (x,u_2)\in B, u\in \pos(u_1,u_2)\},
\end{align*}
where $\pos (u_1,u_2)$ denotes the positive hull of the vectors $u_1$ and $u_2$. The set $A\wedge E$, for $E\in A(d,q),$ is defined in a similar way, by replacing the condition $(x,u_2)\in \Nor M$ by $(x,u_2)\in E\times E^\bot$.

The area measures are projections of the support measures,
$$
S_i(K,\cdot) = \Theta_i (K,\partial K\times \cdot),
$$
where $\partial K$ is the boundary of $K$, and similarly we have $C_i(K,\cdot) = \Theta_i (K,\cdot\times S^{d-1})$. This, however,  does not lead to explicit kinematic formulas for area measures since, for $A\subset \Nor K$ and $B\subset \Nor M$, the set $A\wedge gB$ is not of the form $\partial(K\cap gM)\times C$, for a set $C=C(A,B)\subset S^{d-1}$, in general. Actually, $A\wedge gB$ only involves boundary points of $K\cap gM$ which are common boundary points of $K$ and of $gM$.  

In the following, we investigate such explicit kinematic formulas for area measures. For this purpose, we define $S_d(K,\cdot) = \omega_d V_d(K)\sigma $, where $\sigma$ is the normalized (probability measure) spherical Lebesgue measure on $S^{d-1}$  and $\omega_d=2\pi^{d/2}/\Gamma(d/2)$ (which is the spherical Lebesgue measure of $S^{d-1}$). We shall show that
\begin{equation}\label{PKF5}
\int_{G_d} S_j(K\cap gM, A)\,\mu(d g) = \sum_{k=j}^d  [T_{d,j,k}S_{d+j-k}(K,\cdot)](A) V_{k}(M), 
\end{equation}
for $j= 1,\dots ,d-1,$ and
\begin{equation}\label{CF5} 
\int_{A(d,q)} S_j(K\cap E, A)\,\mu_q(d E) = [T_{d,j,q} S_{d+j-q}(K,\cdot)](A), 
\end{equation}
for $q=2,\ldots,d$ and $j=1,\dots,q-1,$
hold with certain continuous linear operators $T_{d,j,k}$ (depending only on the dimensions $d,j$ and $k$) on the vector space ${\cal M}(S^{d-1})$ of finite signed measures on $S^{d-1}$, supplied with the  weak$^*$ topology. More precisely, we show in Section 3 that $T_{d,j,k}$ is proportional to a composition of the Fourier operators $I_p$ which were studied in \cite{GYY} and \cite{GW} (and a reflection).

In Section 4 we collect some variations and applications of the kinematic formulas.
As a further application of the more explicit version of \eqref{CF5}, we generalize, in the final section, an orthogonality relation for Fourier operators on the sphere, due to Koldobsky \cite{Kol}, from the symmetric to the general (not necessarily symmetric) case.

\section{A local version of Hadwiger's integral theorem}

In this section, we establish the following slight generalization of Hadwiger's general integral geometric theorem (see \cite[Theorem 5.1.2]{SW}). We denote by ${\cal M}^+(S^{d-1})$ the cone of nonnegative measures in ${\cal M}(S^{d-1})$. 

\begin{theorem}\label{thmlocHad} Let $\varphi : {\cal K}'\to {\cal M}^+(S^{d-1})$ be a continuous and additive mapping with $\varphi(\emptyset,\cdot)=0$ (the zero measure). Then, for $K,M\in{\cal K}$ and Borel sets $A\subset S^{d-1}$,
\begin{equation}\label{abstractPKF}
\int_{G_d} \varphi (K\cap gM, A)\,\mu(d g) = \sum_{k=0}^{d}  [T_{d,k}\varphi(K,\cdot)](A) V_{k }(M), 
\end{equation}
with mappings  $T_{d,k} : {\cal M}^+(S^{d-1})\to {\cal M}^+(S^{d-1})$ which are given by the Crofton integrals
\begin{equation}\label{abstractCF} 
T_{d,k}\varphi(K,\cdot) = \int_{A(d,k)} \varphi(K\cap E, \cdot)\,\mu_k(d E) , \quad k=0,\dots, d.
\end{equation}
\end{theorem}

\begin{proof} For the $\mu$-integrability of the integrand in \eqref{abstractPKF}, we reference the discussion in \cite[p. 173]{SW}.

We fix  a real-valued continuous function $f$ on $S^{d-1}$ and define a map $\varphi_f:\mathcal{K}\to\R$ by 
$$
\varphi_f(K)= \int_{S^{d-1}} f(u)\, \varphi (K,du) 
$$ 
for $K\in\mathcal{K}$.  The functional $\varphi_f$ is continuous and additive, since $\varphi$ is continuous and additive. 
For a given $K\in\mathcal{K}$, we then consider the mapping $T_{K,f} : {\cal K}\to \R$ given by
$$
T_{K,f}(M)= \int_{G_d}  \varphi_f (K\cap gM)\,\mu(d g), 
$$
for $M\in{\cal K}$.

In this situation, we can apply \cite[Theorem 5.1.2]{SW} to conclude that 
$$
T_{K,f}(M) = \sum_{k=0}^{d} c_k(K,f) V_k(M)
$$
with constants $c_k(K,f)$ given by 
$$
c_k(K,f) = \int_{A(d,k)} \varphi_f (K\cap E)\, \mu_k(dE) ,
$$
for $k=0,\dots ,d$. 
Using Fubini's theorem, we get
\begin{align*}
\int_{S^{d-1}} &f(u) \left[ \int_{G_d} \varphi (K\cap gM, \cdot)\,\mu(d g)\right]\, (du)\cr
 &= \int_{S^{d-1}} f(u) \left[\sum_{k=0}^{d}  [T_{d,k}\varphi(K,\cdot)] V_{k }(M)\right]\, (du)
\end{align*}
with measures $T_{d,k}\varphi(K,\cdot)$
given by 
$$T_{d,k}\varphi(K,\cdot) = \int_{A(d,k)} \varphi(K\cap E, \cdot)\,\mu_k(d E)$$
Since this holds for all continuous functions $f$, the proof is completed.
\end{proof}

\section{The kinematic formulas}

The purpose of this section is to give a proof of the kinematic formulas \eqref{PKF5} and \eqref{CF5} for area measures. In order to describe the transform $T_{d,j,k}$ in more detail, we first briefly recall the definition and properties of the Fourier operators $I_p,$ for $p=-1,0,1,\dots,d,$ (for details, we refer to \cite{GYY} and \cite{GW}). For a $C^\infty$-function $f$ on $S^{d-1}$ and $p\in\mathbb Z$, let $f_p$ be the homogeneous degree $-d+p$ extension of $f$ to
$\mathbb R^d\setminus \{0\}$. The distributional Fourier transform of
this is denoted by $\hat f_p$. For $0<p<d$, the restriction of
$\hat f_p$ to $S^{d-1}$ is again a smooth function. The mapping $I_p : f\mapsto \hat f_p|_{S^{d-1}}$ intertwines the group action of $SO(d)$, hence it acts as a multiple of the identity on the spaces $H_n^d$, $n=0,1,\dots$, of spherical harmonics. For even $n$, this multiple is real whereas, for odd $n$, it is purely imaginary. In fact, if we denote the multiples by $\lambda_n(d,p)$, we have, for $0<p<d$, 
\begin{equation}\label{multipliers}
\lambda_n(d,p)=\pi^{d/2}2^p(-1)^{n/2}\frac{\Gamma(\frac{n+p}2)}{\Gamma(\frac{n+d-p}2)},\qquad n=0,1,\dots;
\end{equation}
see \cite{GYY}, for example. In the following, we will  use the composition of two of these mappings for different values of $p$. Thus, we will have a mapping $I_pI_q:C^\infty(S^{d-1})\to C^\infty(S^{d-1})$ say, which can be defined in terms of its action, by multiplication, on the spaces $H_n^d\subset C^\infty(S^{d-1})$. Now, the multipliers are real for both the even and odd harmonics. 
The operators $I_p$ can be extended to values of $p$ beyond the integers in the interval $(0,d)$, by analytic continuation of the gamma functions in \eqref{multipliers}. In particular, we will use the operator  $I_{-1}$, which acts on $C^\infty$-functions without (non-trivial) linear part. For $1\le p\le d-1$, the linear operator $I_p$ is bijective and the inverse is (up to a reflection) a multiple of the operator $I_{d-p}$. More precisely,
\begin{equation}\label{inversion}
 I_{d-p}I_p=(2\pi)^dI^*,
\end{equation}
 where $(I^*f)(u)=f(-u)$. Also, since the operators $I_p$ are self-adjoint, they can be extended, by duality, to act on distributions and thus, in particular, on measures $\rho\in{\cal M}(S^{d-1})$. 
 
For $K\in\mathcal{K}$ and $k\in \{1,\ldots,d\}$,  the $k$th {\em mean section body} $M_k(K)$ of $K$ is defined by
$$
h(M_k(K),\cdot)=\int_{A(d,k)}h(K\cap E,\cdot)\, \mu_k(dE),
$$
where $h(K,\cdot)$ is the support function of $K$. The linearity of the first area measure then implies that
\begin{equation}\label{area1MSB}
S_1(M_k(K),\cdot)=\int_{A(d,k)}S_1(K\cap E,\cdot)\, \mu_k(dE).
\end{equation}
For $k\in\{2,\dots, d\}$ and a convex body $K$ with $\dim K \ge d+2-k$, it follows 
from Corollary 3.5 in \cite{GW} that the area measures of $K$ and $-M_k(K)$ are connected by the relation
\begin{equation}\label{meansec}
S_{d+1-k}(K,\cdot) = c(d,k) I_{d-1}I_{d+1-k}S_1(-M_k(K),\cdot)
\end{equation}
with
\begin{equation}\label{constant} c(d,k)=\frac{(d-1)(d+1-k)\Gamma\left(\frac{k}{2}\right)}{(k-1)2^{2d-k}\pi^{\frac{3d-k}{2}}\Gamma\left(\frac{d}{2}
\right)}.
\end{equation}

We now formulate \eqref{CF5} in a more precise version.

\begin{theorem}\label{surfCrofton} For $1\le j<q\le d$ and $K\in {\cal K}$, we have
\begin{equation}\label{croftform}
\int_{A(d,q)}S_j(K\cap E,\cdot)\,\mu_{q}(dE)=a(d,j,q)I_{j}I_{q-j}S_{d+j-q}(-K,\cdot)
\end{equation}
with
\begin{equation}\label{coefficients}
a(d,j,q) = \frac{j}{2^d\pi^{(d+q)/2}(d+j-q)}\frac{\Gamma (\frac{q+1}{2})\Gamma({d-j})}{\Gamma (\frac{d+1}{2})\Gamma({q-j})} .
\end{equation}
\end{theorem}

\begin{proof} For $q=d$ the assertion of the theorem is true, since $a(d,j,d) I_j I_{d-j} I^*$ is the identity operator. Hence we consider the 
case $q\le d-1$ in the following. 
Moreover, we may assume $\dim K=d$, since the general case can then be obtained by approximation (using the weak continuity of area measures). In this case, for $\mu_q$-almost all $E\in A(d,q)$ such that $K\cap E\not=\emptyset$, we have $\dim (K\cap E) = q\ge j+1$. 
From \eqref{area1MSB} and \eqref{meansec}, we therefore get
\begin{align*}
\int_{A(d,q)}&S_{j}(-(K\cap E),\cdot)\,\mu_{q}(dE)\cr
&=c(d,d+1-j)\int_{A(d,q)}I_{d-1}I_j S_1(M_{d+1-j}(K\cap E),\cdot)\,\mu_{q}(dE)\cr
&=c(d,d+1-j)I_{d-1}I_j\int_{A(d,q)}\int_{A(d,d+1-j)}S_1(K\cap E\cap F),\cdot)\cr
&\quad\times\,\mu_{d+1-j}(dF)\,\mu_{q}(dE)\cr
&=c(d,d+1-j)I_{d-1}I_j\int_{A(d,q+1-j)}S_1(K\cap H,\cdot)\,\psi(dH),
\end{align*}
where $\psi$ is the locally finite image measure (on $A(d,q+1-j)$) of $\mu_{d+1-j}\otimes\mu_{q}$ under the (almost everywhere defined) mapping $(E,F)\mapsto E\cap F$. We have $g(E\cap F) = gE\cap gF$ and so $\psi$ is motion invariant and hence a multiple $c\mu_{q+1-j}$ of $\mu_{q+1-j}$. Therefore,
\begin{align*}
\int_{A(d,q)}&S_{j}(-(K\cap E),\cdot)\,\mu_{q}(dE)\cr
&=c(d,d+1-j)cI_{d-1}I_j\int_{A(d,q+1-j)}S_1(K\cap H,\cdot)\,\mu_{q+1-j}(dH)\cr
&= c(d,d+1-j)cI_{d-1}I_j S_1(M_{q+1-j}(K),\cdot)\\
&= c\frac{c(d,d+1-j)}{(2\pi)^dc(d,q+1-j)}I_jI_{q-j} S_{d+j-q}(K,du), 
\end{align*}
where  we have used \eqref{area1MSB} and \eqref{meansec} again (observe that $\dim K = d\ge d+j-q+1$), as well as \eqref{inversion} for the last equality. The constant $c$ can be determined by an explicit calculation via the Crofton formula for intrinsic volumes. In fact, using 
\cite[Theorem 5.1.1]{SW} and the notation introduced there, we get
\begin{align*}
\int_{A(d,q)}&\int_{A(d,d+1-j)} V_0(E\cap F\cap B^d)\,  \mu_{d+1-j}(dF)\,\mu_{q}(dE)\cr
&=c_{0,d}^{d+1-j,j-1}\int_{A(d,q)}V_{j-1}(E\cap B^d)\, \mu_{q}(dE)\cr
&=c_{0,d}^{d+1-j,j-1}c_{j-1,d}^{q,d+j-1-q}V_{d+j-1-q}(B^d)
\end{align*}
as well as
\begin{align*}
\int_{A(d,q+1-j)}& V_0(H\cap B^d)\,  \mu_{q+1-j}(dH)\cr
&=c_{0,d}^{q+1-j,d+j-1-q}V_{d+j-1-q}(B^d).
\end{align*}
Hence
$$c= \frac{c_{0,d}^{d+1-j,j-1}c_{j-1,d}^{q,d+j-1-q}}{c_{0,d}^{q+1-j,d+j-1-q}} = c^{q,d+1-j}_{d,q+1-j} .
$$
Inserting the latter value from \cite[(5.5)]{SW} and using  \eqref{constant}, we get the value of $a(d,j,q)$. (It should be remarked that formula (5.5) in \cite{SW} is based on formula (5.4) which uses factorials; in the corresponding expression with Gamma functions given at the top of page 172 in \cite{SW}, a factor is missing which however cancels out in our situation.)
\end{proof}

Theorem \ref{surfCrofton} shows that the linear operator $T_{d,j,q}$ used in \eqref{CF5} and \eqref{PKF5} is given by 
$$
T_{d,j,q}=a(d,j,q) I_j I_{q-j} I^*
$$ 
for $1\le j<q\le d$, and $T_{d,j,d}$ is the identity operator.

We next state and prove a more specific version of \eqref{PKF5}. In fact, we just apply Theorem 2.1 to the mapping $\varphi : K\mapsto S_j(K,\cdot)$. It is well-known that $S_j(K,\cdot)$ depends continuously and additively on $K$. The Crofton integrals can be expressed using Theorem 3.1. 

\begin{theorem}\label{PKF6} For $1\le j \le d-1$, $K,M\in {\cal K}$ and Borel sets $A\subset S^{d-1}$, we have
\goodbreak
\begin{align}
\int_{G_d}&S_j(K\cap gM,A)\,\mu(dg)=\binom{d-1}{ j}^{-1}\omega_{d-j} V_d(K)\sigma (A)V_j(M)\label{PKF7}\\
& + \sum_{k=j+1}^{d-1} a(d,j,k)[I_{j}I_{k-j}S_{d+j-k}(-K,\cdot)](A)V_k(M) + S_j(K,A)V_d(M),\nonumber
\end{align}
with constants $a(d,j,k)$ given by \eqref{coefficients}. 
\end{theorem}

\begin{proof} The assertion of the theorem follows by combining Theorem \ref{thmlocHad} and Theorem \ref{surfCrofton}. In addition, we only have to observe that the Crofton integrals
$$
T_{d,k}S_j(K,\cdot) = \int_{A(d,k)} S_j(K\cap E, \cdot)\,\mu_k(d E) , \quad k=0,\dots, d,
$$
which appear in Theorem 2.1, vanish for $k=0,...,j-1$ and that $S_j(K\cap E,\cdot)$, for $E\in A(d,j)$ is proportional to $V_j(K\cap E)$ times the (normalized) spherical Lebesgue measure $\sigma_{E^\bot}$ in $E^\bot$. Integration over all $E$ parallel to the same subspace $L$ then yields
\begin{align*}
T_{d,j}S_j(K,\cdot) &= \binom{d-1}{ j}^{-1}\omega_{d-j} V_d(K)\int_{G(d,j)} \sigma_{L^\bot}\,\nu_k(d L)\\
& =\binom{d-1}{ j}^{-1} \omega_{d-j} V_d(K)\, \sigma .
\end{align*}
Moreover
$$
T_{d,d}S_j(K,\cdot) = a(d,j,d)I_{j}I_{d-j}S_{j}(-K,\cdot) = S_j(K,\cdot)
$$
by the inversion formula \eqref{inversion}.
\end{proof}

\section{Some special cases, variations, and consequences}\theoremstyle{definition}

In this section, we collect some remarks on the kinematic formulas including variations and applications.

\begin{rem} We emphasize the special case $j=d-1$ of \eqref{PKF7}. It yields
\begin{equation}\label{d-1-case}
\int_{G_d}S_{d-1}(K\cap gM,A)\,\mu(dg)=2V_d(K)\sigma (A)V_{d-1}(M)
 + S_{d-1}(K,A)V_d(M) .
\end{equation}
This formula has a translative counterpart, namely
\begin{equation}\label{d-1-casetrans}
\int_{\R^d}S_{d-1}(K\cap (M+x),A)\,\lambda(dx)=V_d(K)S_{d-1}(M,A) + S_{d-1}(K,A)V_d(M) ,
\end{equation}
where $\lambda$ denotes the Lebesgue measure in $\R^d$. From \eqref{d-1-casetrans}, the kinematic formula \eqref{d-1-case} follows directly by integrating with the invariant probability measure $\nu$ over the rotation group $SO(d)$ and using Fubini's theorem,
\begin{align*}
\int_{SO(d)}S_{d-1}(\rho M,\cdot)\, \nu(d\rho)&=\int_{SO(d)}\int_{S^{d-1}}\mathbf{1}\{\rho u\in\cdot\}S_{d-1}(M,du)\, \nu(d\rho)\\
&=\int_{S^{d-1}} \sigma(\cdot) \, S_{d-1}(M,du)\\
&=2V_{d-1}(M) \sigma(\cdot) .
\end{align*}
Formula \eqref{d-1-casetrans} makes several appearances; see, for example, \cite[(3.4)]{W97} or \cite[(7.17)]{SW2000}. It can be deduced from a general (and simple) translation formula for $\sigma$-finite measures on $\R^d$ (Theorem 5.2.1 in \cite{SW}) but also follows from more general results on mixed functionals from Translative Integral Geometry (see \cite[Section 6.4]{SW}). 
\end{rem}

\begin{rem}
We can obtain a variant of \eqref{croftform} from \eqref{inversion} in the form
\begin{equation}\label{croftform2}
\int_{A(d,q)}I_{d+j-q}I_{d-j}S_j(K\cap E,\cdot)\,\mu_{q}(dE)=a'(d,j,q)S_{d+j-q}(-K,\cdot),
\end{equation}
where $a'(d,j,q)=a(d,j,k)(2\pi)^{2d}$, 
which yields a Crofton-type representation of $S_{d+j-q}(-K,\cdot)$.
\end{rem}

\begin{rem}\label{lifting}
Although we mentioned in the introduction that a Crofton formula with integrand $S'_j(K\cap E,\cdot)$, the area measure of $K\cap E$ calculated in the subspace $L(E)$, is not possible in a direct manner, we can obtain a version of \eqref{croftform} with $S_j(K\cap E,\cdot)$ replaced by $S'_j(K\cap E,\cdot)$ if we use an appropriate lifting for measures from the unit sphere in $L(E)$ to  $S^{d-1}$ (see \cite{GKW}). Namely, it was shown in \cite[Theorem 6.2]{GKW} that 
$$
S_j(K\cap E,\cdot) = c(d,q,j) \pi^*_{L(E),-j} S_j'(K\cap E,\cdot)
$$
with a given constant $c(d,q,j)$.
\end{rem}

\begin{rem}
It is, of course, possible to consider the Crofton-type integral
\begin{equation}\label{relCF}
\int_{A(d,q)} S_{j}'(K\cap E,A\cap L(E) ) \,\mu_q(dE)
\end{equation}
for $j\in\{ 1,\dots,q-1\}, K\in{\cal K}$ and $A\subset S^{d-1}$. As we argued in the introduction, the result will not be a multiple of $S_{d+j-q}(K,A)$. In \cite{Rataj99}, translative and kinematic Crofton formulas of this kind are derived (not only for area measures, but more generally for support measures). In general, the integrals cannot be expressed as a simple transform of an area (or support) measure but are more complicated curvature expressions of $K$.

In the case $j=q-1$, the translative formula for $K\in\cal K$ and $L\in G(d,q)$ gives rise to the {\it relative Blaschke section body} $B_L(K)\subset L$ of $K$, which was defined in \cite{GKW98},
\begin{equation}\label{relCF2}
S_{q-1}'(B_L(K),\cdot) = \int_{L^\perp} S_{q-1}'((K+x)\cap L,\cdot ) \,\lambda_{L^\perp}(dx) 
\end{equation}
(here, $\lambda_M$ denotes the Lebesgue measure in $M\in G(d,i)$)
and \eqref{relCF} yields the area measure of the {\it Blaschke section body} $B_q(K)$,
$$
S_{d-1}(B_q(K),A) = \int_{A(d,q)} S_{q-1}'(K\cap E,A\cap L(E) ) \,\mu_{q}(dE) .
$$
Using the spherical projection $\pi_{L,1}$ from \cite{GKW}, one obtains
$$
S_{q-1}'(B_L(K),\cdot) = \pi_{L,1} S_{d-1}(K,\cdot)
$$
and therefore
\begin{align*}
S_{d-1}(B_q(K),\cdot) &= \int_{G(d,q)} \pi_{L,\infty}^*\pi_{L,1} S_{d-1}(K,\cdot)\, \nu_q(dL)\\
&=\pi^{(q)}_{\infty,1} S_{d-1}(K,\cdot) 
\end{align*}
(see \cite[Theorem 3.1]{GKW98}, \cite[(3.4)]{Rataj99} and \cite[(6.7)]{GKW}). Here, $\pi_{L,\infty}^*$ is the trivial lifting (of a measure) from $L\cap S^{d-1}$ to $S^{d-1}$ and $\pi^{(q)}_{\infty,1}$ is the corresponding {\it mean lifted projection} (\cite[Definition 7.3]{GKW}).
\end{rem}

\begin{rem}
The Crofton formula \eqref{croftform} also yields a connection between the Fourier operators $I_p$ and spherical projections and liftings. In fact, \eqref{relCF2}
is easily seen to be equivalent to
$$
S_{q-1}(B_L(K),\cdot) = \int_{L^\perp} S_{q-1}(K\cap (L+x),\cdot ) \, \lambda_{L^\perp}(dx) 
$$
(for example, by using the lifting from Remark \ref{lifting}). Integrating over the rotation group $SO(d)$, we obtain the {\it Blaschke section body (of the second kind)} $\tilde B_q(K)$, defined in \cite{Kiderlen} and studied further in \cite[Example 4]{GKW},
\begin{align*}
S_{d-1}( \tilde B_q(K),\cdot) &= \int_{SO(d)}\int_{L^\perp} S_{q-1}(K\cap \vartheta(L+x)) \,\lambda_{L^\perp}(dx) \,\nu(d\vartheta)\\
&= \int_{A(d,q)} S_{q-1}(K\cap E) \,\mu_q(dE) \\
&= a(d,q-1,q) I_{q-1}I_1 S_{d-1}(-K,\cdot) ,
\end{align*}
here we used \eqref{croftform}. A comparison of this result with a formula on p. 28 of \cite{GKW} now shows that
$$
I_{q-1}I_1I^\ast = a(d,q) \pi^{(q)}_{1,1-q}
$$
with 
$$a(d,q)= \binom{d-1}{ k-1}^{-1}a(d,q-1,q)$$ 
and where the mean lifted projection $\pi^{(q)}_{1,1-q}$ is defined in \cite[(7.1)]{GKW}. From results of Kiderlen \cite{Kiderlen}, it follows that the operator $\pi^{(q)}_{1,1-q}$ is injective on centered measures on the sphere. This is now also apparent from the injectivity properties of the $I_p$ operators. More generally, from Theorem \ref{surfCrofton} we see that, for each $1\le j<q\le d$, the Crofton integral 
$$
\int_{A(d,q)} S_{j}(K\cap E,\cdot ) \,\mu_q(dE)
$$
determines a body $K$ of dimension $\ge d+j-q+1$ uniquely. Kiderlen applied  the case $j=q-1$ to a stereological problem for random sets $Z$ or particle processes $X$ in $\R^d$, namely for the estimation of the mean normal measure of $Z$ and $X$ from measurements in $q$-dimensional sections. We now see, that a similar stereological application is possible, using lower order area measures of sections of $Z$ or $X$.
\end{rem}

\begin{rem} If $f_s\in H^d_s$, then 
$$
I_j I_{q-j} f_s=\lambda_s(d,j)\lambda_s(d,q-j) f_s=(-1)^sb_s(d,j,q) f_s
$$
with 
$$
b_s(d,j,q)=2^q\pi^d\frac{\Gamma\left(\frac{s+j}{2}\right)\Gamma\left(\frac{s+q-j}{2}\right)}{\Gamma\left(\frac{s+d-j}{2}\right)\Gamma\left(\frac{s+d-q+j}{2}\right)}.
$$
Hence, in this case, Theorem \ref{surfCrofton} becomes
\begin{align*}
&\int_{A(d,q)}\int_{S^{d-1}}f_s(u)\, S_j(K\cap E,du)\, \mu_q(dE)\\
&\qquad =a_s(d,j,q)\int_{S^{d-1}} f_s(u)\, S_{d+j-q}(K,du),
\end{align*}
where $a_s(s,j,q)=a(d,j,q)b_s(d,j,q)$. 

As indicated in \cite{BH}, this version of Theorem \ref{surfCrofton} immediately implies a Crofton formula for translation invariant tensor valuations, stated as Corollary 6.1 in \cite{BH} (and hence also \cite[Theorem 3]{BH}). 
To see this, note that the translation invariant tensor valued valuation $\Psi_{k,s}$, defined in \cite[Proposition 4.16]{BH}, can 
be written in the form
$$
\Psi_{k,s}(K)=\frac{\Gamma\left(\frac{d-k+s}{2}\right)}{\Gamma\left(\frac{d-k}{2}\right)}\frac{1}{\pi^{\frac{s}{2}}s!}\binom{d-1}{k}\frac{1}{\omega_{d-k}}
\int_{S^{d-1}}\bar f_s(u)\, S_k(K,du)
$$
with a function $\bar f_s$ on $S^{d-1}$ which takes values in the vector space of symmetric tensors of rank $s$ and whose coefficients with respect 
to a  basis of this vector space are spherical harmonics of degree $s$; see the proof of \cite[Corollary 4.17]{BH}.
\end{rem}

\begin{rem}
In \cite[Theorem 1]{BH}, the Fourier transforms $\mathbb{F}$ of spherical valuations have been determined. This result can now be expressed 
in terms of the operators $I_p$. For $K\in \mathcal{K}$, $f\in L^2(S^{d-1})$ and $k\in \{1,\ldots,d-1\}$, let 
$$
\bar\mu_{k,f}(K)=\binom{d-1}{k}(2\pi)^{\frac{k}{2}}\int_{S^{d-1}}f(u)\, S_k(K,du),
$$
which is a convenient renormalization of the spherical valuation $\mu_{k,f}$ introduced in \cite{BH}. Then \cite[Theorem 1]{BH} 
can be expressed in the equivalent form 
$$
\mathbb{F}(\bar\mu_{k,f})=(2\pi)^{-\frac{d}{2}}\, \bar\mu_{d-k,I_k f},
$$
for $f\in H^d_s$, which then extends to arbitrary $f\in L^2(S^{d-1})$. Hence, the Fourier transform of the spherical valuation associated with a function is up to the factor $(2\pi)^{-\frac{d}{2}}$ the spherical valuation of an $I_p$ transform of that function.
\end{rem}

\section{Orthogonality aspects}
It is a classical result of Fourier analysis that, if $f$ is an even function on $\mathbb R^d$ and has appropriate integrability properties, then its Fourier transform $\hat f$ satisfies, for any $k=1,\dots,d-1$ and any $L\in G(d,d-k)$,
$$
(2\pi)^{k}\int_{L} f(x)\,\lambda_L (dx)=\int_{L^\perp}\hat f(x)\,\lambda_{L^\bot} (dx),
$$
see, for example, \cite[Lemma 3.24]{Kol}. 
We will be interested in a spherical analogue of this result, due to Koldobsky \cite[Lemma 3.25]{Kol}. In our notation, Koldobsky's result is
\begin{equation}\label{spherorthog}
\int_{S^{d-1}\cap H^\perp}(I_{d-k}f)(u)\,\sigma_{H^\perp}(du)\\=\pi^{\frac{d}{2}}2^{d-k}\displaystyle\frac{\Gamma(\frac{d-k}{2})}{\Gamma(\frac{k}{2})}\int_{S^{d-1}\cap H}f(u)\,\sigma_H(du)
\end{equation}
for all even $f\in C^\infty(S^{d-1})$ and $H\in G(d,k)$; see also \cite[Theorem 2.7]{Mil}. Here and in the following, $\sigma_H$ denotes the normalized (probability measure) spherical Lebesgue measure on $S^{d-1}\cap H$.  It is our intention, in this section, to see how (\ref{spherorthog}) is a consequence of Theorem \ref{surfCrofton} and to use this theorem to find generalizations.

For this it will be convenient to recall the notion of Radon transforms between Grassmann manifolds, see, for example, \cite{GSW}. For $1\le i,j\le d-1$, we denote by $R_{i,j}:C(G(d,i))\to C(G(d,j))$ the Radon transform given by
\begin{align*}
(R_{i,j}f)(E)&=\int_{G(E,i)}f(L)\,\nu_i^E( dL)\qquad \text{for }E\in G(d,j),\ f\in C(G(d,i));
\intertext{where}
G(E,i)&=\begin{cases}\{L\in G(d,i):L\subset E\}&\text{if }i<j;\\
\{L\in G(d,i):L\supset E\}&\text{if }i>j;\\
\{E\}&\text{if }i=j.\end{cases}
\end{align*}
and the integration is with respect to the invariant probability measure $\nu_i^E$ on $G(E,i)$. This transform can be extended to $L^2(G(d,i))$ since it intertwines the action of the rotation group and therefore acts as a multiple of the identity on the irreducible invariant subspaces. In this case,  it is defined by means of its action on these subspaces. Using this notation, Koldobsky's result (\ref{spherorthog}) becomes
\begin{equation}\label{Grassorthog}
(R_{1,d-k}I_{d-k}f)(H^\perp)=c_{d,k}(R_{1,k}f)(H)
\end{equation}
with
$$
c_{d,k}=\pi^{\frac{d}{2}}2^{d-k}\displaystyle\frac{\Gamma(\frac{d-k}{2})}{\Gamma(\frac{k}{2})}
$$
for $f\in C^\infty(G(d,1)),\ H\in G(d,k)$. Here, we identify functions on $G(d,1)$ with even functions on $S^{d-1}$. 
In the following, for a function $f$ on $G(d,i)$, we will denote by $f^\perp$ the function on $G(d,d-i)$ given by $f^\perp(L)=f(L^\perp)$. Using this notation, it is well known that, for even $f\in C^\infty(S^{d-1}),\ 
(R_{i,j}f)^\perp=R_{d-i,d-j}f^\perp$ and, moreover, the special case $k=d-1$ of \eqref{Grassorthog} states
\begin{equation}\label{speccase}
I_1f=\frac{2\pi^{\frac{d+1}{2}}}{\Gamma\left(\frac{d-1}{2}\right)}\,(R_{1,d-1}f)^\perp=
\frac{2\pi^{\frac{d+1}{2}}}{\Gamma\left(\frac{d-1}{2}\right)}\, R_{d-1,1}f^\perp.
\end{equation}
Furthermore, for $1\le i\le j\le k\le d-1$, we have
$$
R_{i,k}=R_{j,k}R_{i,j}\qquad\text{and}\qquad R_{d-i,d-k}=R_{d-j,d-k}R_{d-i,d-j}.
$$
For $j=2,\dots,d-2$, the operators $R_{j,1}:C^\infty(G(d,j))\to C^\infty(G(d,1))$ are not injective, however, their restrictions to the range of $R_{1,j}$ (or to the range of a composition $R_{k,j}R_{1,k}$, for any $k=j+1,\dots,d-1$) are injective. Thus, 
(\ref{Grassorthog}) is equivalent to
$$
R_{d-k,1}R_{1,d-k}I_{d-k}f=c_{d,k}R_{d-k,1}R_{d-1,d-k}f^\perp=c_{d,k}R_{d-1,1}f^\perp,
$$
or (using \eqref{speccase})
\begin{equation}\label{Grassorthog2}
R_{d-k,1}R_{1,d-k}f=\tilde c_{d,k}I_1I_kf
\end{equation}
with
$$
\tilde c_{d,k}=2^{-(k+1)}\pi^{-d}\frac{\Gamma\left(\frac{d-k}{2}\right)
\Gamma\left(\frac{d-1}{2}\right)}{\Gamma\left(\frac{k}{2}\right)\Gamma\left(\frac{1}{2}\right)}
$$
for even $f\in C^\infty(S^{d-1})$ and $k=1,\dots,d-1$.
The case $k=d-1$ is just the statement that, for even functions, $I_{d-1}=(2\pi)^d I_1^{-1}$. For
$k = 1$, we obtain the previously noted connection between $I_1$ and $R_{1,d-1}$, see \eqref{speccase}. 

We will now prove a more general formula which contains \eqref{Grassorthog2} as the case 
$q=j+1=k+1$.

\begin{theorem}\label{Kolspherical} For $1\le j<q<d$ and even $g\in C^\infty(S^{d-1})$, we have 
\begin{equation}\label{Grassorthog3}
R_{1,q-j}I_{q-j}I_jg=c_{d,q,j}R_{d-j,q-j}R_{1,d-j}g 
\end{equation}
with
$$
c_{d,q,j} = 2^q\pi^d\frac{\Gamma(\frac{j}{2})\Gamma(\frac{q-j}{2})}{\Gamma(\frac{d-j}{2})\Gamma(\frac{d-q+j}{2})}.
$$
\end{theorem}
\begin{proof}
We will require some more notation for the proof. For an affine space $E\in A(d,q)$ and convex body $K\subset \mathbb R^d$, we denote by $L(E)$ and $L(K)$, respectively, the subspace parallel to $E$ and the subspace parallel to the affine hull of $K$. We will also need the subspace determinant $[L_1,L_2]$ for linear spaces $L_1,L_2$ and the generalized cosine $\langle L_1,L_2\rangle$ when the spaces have the same dimension. The definitions can be found in \cite[pages 597--598]{SW}. In particular, we note that $|\langle L_1,L_2\rangle|=[L_1,L_2^\perp]=[L_1^\perp,L_2]$.

Let $1\le j<q\le d-1$ and let $K$ be a convex body with $\dim K=d-q+j$. We, first, note that, for $\mu_q$-almost all $E\in A(d,q)$ with $K\cap E\ne\emptyset$, we have $\dim(K\cap E)=j$.
 Thus, for such an $E$ and any even function $f\in C(S^{d-1})$, we get
$$
\int_{S^{d-1}}f(u)\,S_j(K\cap E,du)=c_{1}V_j(K\cap E)(R_{1,d-j}f)((L(K)\cap L(E))^\perp)
$$
with some constant $c_1=c_1(d,j)$. All constants $c_2,c_3,\ldots$ in the proof will only depend on $d,j$ and possibly on $q$. 

The Cavalieri principle then gives, for $\nu_{q}$-almost all $L\in G(d,q)$,
$$
\int_{L^\perp}V_j(K\cap(L+x))\,\lambda_{L^\perp}(dx)=|\langle (L\cap L(K))^\perp\cap  L(K),L^\perp\rangle| V_{d-q+j}(K).
$$
The Blaschke-Petkantschin formulas (see, for example, \cite[Theorem 7.2.6]{SW}) show that,  for fixed $F\in G(d,d-q+j)$ and any $h\in C(G(d,q))$,
\begin{equation*}
\int_{G(d,q)}h(L)\,\nu_{q}(dL)\\
=c_2\int_{G(F,j)}\int_{G(M,q)}h(L)[L,F]^j\,\nu_{q}^M(dL)\,\nu_j^F(dM).
\end{equation*}
 Consequently, for even $g\in C^\infty(S^{d-1})$,
\begin{align*}
\int_{A(d,q)}&\int_{S^{d-1}}g(u)\,S_j(K\cap E,du)\,\mu_{q}(dE)\\
&=c_{1}\int_{G(d,q)}\biggl((R_{d-1,j}g^\perp)(L(K)\cap L)\\
&\hskip20pt\times\int_{L^\perp}V_j(K\cap(L+x))\,\lambda_{L^\perp}(dx)\biggr)\,\nu_{q}(dL)\\[5pt]
&=c_{1}V_{d-q+j}(K)\int_{G(d,q)}\biggl(|\langle(L(K)\cap L)^\perp\cap L(K),L^\perp\rangle|\\
&\hskip20pt \times(R_{d-1,j}g^\perp)(L(K)\cap L)\biggr)\,\nu_{q}(dL)\\[5pt]
&= c_{3}V_{d-q+j}(K)\int_{G(L(K),j)}\int_{G(M,q)} (R_{d-1,j}g^\perp)(L(K)\cap L)\\
&\hskip20pt \times  [L,L(K)] [L,L(K)]^j\, 
\nu_{q}^M(dL) \,\nu_j^{L(K)}(dM)\\[5pt]
&= c_{3}V_{d-q+j}(K)\int_{G(L(K),j)} (R_{d-1,j}g^\perp)(M)\\
&\hskip20pt\times\int_{G(M,q)}[L,L(K)]^{j+1}\, 
\nu_{q}^M(dL) \,\nu_j^{L(K)}(dM)\\[5pt]
&=c_{4}V_{d-q+j}(K)(R_{j,d-q+j}R_{d-1,j}g^\perp)(L(K))\\
&=c_{4}V_{d-q+j}(K)(R_{d-j,q-j}R_{1,d-j}g)(L(K)^\perp).
\end{align*}
Finally, note that
$$
S_{d-q+j}(K,\cdot)=S_{d-q+j}(-K,\cdot)=\binom{d-1}{j}^{-1}\omega_{q-j}V_{d-q+j}(K)\sigma_{L(K)^\perp},
$$
and so the right-hand side of the Crofton formula (Theorem \ref{surfCrofton}), applied to the even function $g$, yields
\begin{align*}
V_{d-q+j}&(K)(R_{1,q-j}I_jI_{q-j}g)(L(K)^\perp)\\
&=V_{d-q+j}(K)\int_{S^{d-1}}(I_jI_{q-j}g)(u)\,\sigma_{L(K)^\perp}(du)\\
&=\binom{d-1}{j}\omega_{q-j}^{-1}\int_{S^{d-1}}(I_jI_{q-j}g)(u)\, S_{d-q+j}(K,du)\\
&=c_5  \int_{A(d,q)}\int_{S^{d-1}} g(u)\, S_{j}(K\cap E,du)\, \mu_q(dE)\\
&=c_{d,j,q}V_{d-q+j}(K)(R_{d-j,q-j}R_{1,d-j}g)(L(K)^\perp),
\end{align*}
and hence $R_{1,q-j}I_{q-j}I_j g=R_{d-j,q-j}R_{1,d-j}g$.  
The value of $c_{d,q,j}$ comes from (3.1) and letting $g$ be constant. 
\end{proof}

Although (\ref{Grassorthog3}) appears to be a generalization of (\ref{Grassorthog2}), the difference between the two is mostly in the formulation. To see this, note that (\ref{Grassorthog2}) can be rephrased as saying that any two members of the family of operators $R_{d-k,1}R_{1,d-k}I_{d-k}$, for $k=1,\dots,d-1$, are multiples of each other. Thus, for $1\le j<q<d$, (\ref{Grassorthog2}) gives
$$
R_{q-j,1}R_{1,q-j}I_{q-j}=cR_{d-j,1}R_{1,d-j}I_{d-j}=cR_{q-j,1}R_{d-j,q-j}R_{1,d-j}I_{d-j}
$$
and so the injectivity results mentioned above (parenthetically) yield (\ref{Grassorthog3}). The purpose of Theorem \ref{Kolspherical} is not so much to find a generalization of (\ref{Grassorthog2}) as to find a proof (in this case using Crofton formulas) which allows extension to functions which are not necessarily even. This will be achieved by replacing the low dimensional body $K$ in the proof of Theorem \ref{Kolspherical} with a full dimensional polytope. The nature of the surface area measures of polytopes will lead us to consider the asymmetric $p$-cosine transform. For a subspace $M\in G(d,j)$ and appropriately chosen $p\in\mathbb Z$, this transform $H_p^M:C^\infty(S^{d-1}\cap M)\to C^\infty(S^{d-1}\cap M)$ is defined by
$$
(H_p^Mf)(u)=\int_{S^{d-1}\cap M\cap u^+}\langle u,v\rangle^pf(v)\sigma_M(dv),\qquad u\in S^{d-1}\cap M,
$$
here, $u^+$ denotes the half space comprising those $x\in\mathbb R^d$ with $\langle x,u\rangle>0$. In view of the possibility of negative values of $p$, care has to be taken over integrability issues in the above formula. The asymmetric (and symmetric) $p$-cosine transforms have been used in many different contexts, most recently it has emerged as an important tool in the study  of non-symmetric convex and star-shaped bodies. In our situation, it arises through the study of surface area measures of polytopes and provides a link to the symmetric case via the connection with the $I_p$ operators. This connection is clear in the work of Gelfand and Shilov \cite{GelShi} on Fourier transforms of homogeneous distributions. In particular, it follows from their work that, for even $p\ge 0$ and odd $f\in C^\infty(S^{d-1}),\ H_pf=cI_{-p}f$ (here we suppress the superscript $M$ in case $M=\mathbb R^d$); see, for example \cite{GYY}. For non-integer values of $p$, this connection is easy to establish because the poles of the Gamma function are avoided. Also the injectivity properties of $H_p$ can usually be deduced from those of the Fourier transform. The eigenspaces are again the spherical harmonics and the eigenvalues are known, see, for example Rubin \cite{Rub1, Rub2}. Recent applications to the study of $L_p$ intersection bodies and affine isoperimetric inequalities can be found in the work of Haberl \cite{Hab} and Haberl and Schuster \cite{HabSch} and the references therein.

We now turn to the extension of (\ref{Grassorthog}) to non-symmetric functions. As indicated above, we will apply Theorem \ref{surfCrofton} to a $d$-dimensional polytope $P$. Let $\mathcal{F}_{d-1}(P)$ denote the set of $(d-1)$-dimensional faces (the facets) of $P$, and for any facet $F$ of $P$, let $v_F$ denote its exterior unit normal vector. Then, for $1\le j\le d-1$ and any $L\in G(d,j+1)$, we have
\begin{align*}
\int_{L^\perp}\int_{S^{d-1}}&f(u)\, S_j(P\cap(L+x),du)\,\lambda_{L^\perp}(dx)\\
&=c_6 \sum_{F\in{\mathcal F}_{d-1}(P)}
\int_{L^\perp}V_j(F\cap(L+x))\\
&\quad\times\int_{S^{d-1}\cap(v_F^\perp\cap L)^\perp\cap(v_F|L)^+}f(u)\,\sigma_{(v_F^\perp\cap L)^\perp}(du)\, \lambda_{L^\perp}(dx)\\
&=c_6 \sum_{F\in{\mathcal F}_{d-1}(P)}V_{d-1}(F)|\langle L^\perp,v_F^\perp\cap(v_F^\perp\cap L)^\perp\rangle|\left(H_0^{(v_F^\perp\cap L)^\perp}f\right)(\widehat{v_F|L}),
\end{align*}
where $\widehat{v_F|L}$ is the unit vector in direction $v_F|L$ (provided $v_F|L\not= 0$) and varying constants depending only on $d,j$ are denoted by $c_6,c_7,\ldots$ 

Again, using the Blaschke-Petkantschin formulas, we have
\begin{align*}
\int_{A(d,j+1)}&\int_{S^{d-1}}f(u)\, S_j(P\cap E,du)\, \mu_{j+1}(dE)\\
&=c_{7}\sum_{F\in{\mathcal F}_{d-1}(P)}V_{d-1}(F)\\
&\quad\times\int_{G(v_F^\perp,j)}\int_{G(M,j+1)}[L,v_F^\perp]^{j+1}\left(H_0^{M^\perp}f\right)
(\widehat{v_F|L})\,\nu_{j+1}^M(dL)\,\nu_j^{v_F^\perp}(dM).
\end{align*}
Now, for $M\in G(v_F^\perp,j)$, the manifold $G(M,j+1)$ can be identified with the half-sphere $S^{d-1}\cap M^\perp\cap v_F^+$. In which case
\begin{align*}
\int_{G(M,j+1)}&[L,v_F^\perp]^{j+1} \left(H_0^{M^\perp}f\right)(\widehat{v_F|L})\,\nu_{j+1}^M(dL)\\
&=c_8 \int_{S^{d-1}\cap M^\perp\cap v_F^+}\langle u,v_F\rangle^{j+1}\left(H_0^{M^\perp}f\right)(u)\, 
\sigma_{M^\perp}(du)\\
&=c_8\left(H_{j+1}^{M^\perp}H_0^{M^\perp}f\right)(v_F).
\end{align*}
For any vector $v\in S^{d-1},\ N\in G(\langle v\rangle,d-j)$ and $f\in C^\infty(S^{d-1})$ we define $f_j^v\in C^\infty(G(\langle v\rangle,d-j))$ by
$$
f_j^v(N)=\left(H_{j+1}^{N}H_0^{N}f\right)(v).
$$
Here, $\langle v\rangle\in G(d,1)$ is the line in direction $v$.
Then 
\begin{align*}
\int_{A(d,j+1)}\int_{S^{d-1}}&f(u)\,S_j(P\cap E,du)\,\mu_{j+1}(dE)\\
&=c_9\sum_{F\in{\mathcal F}_{d-1}(P)}V_{d-1}(F)(R_{j,d-1}(f_j^{v_F})^\perp)(v_F^\perp).
\end{align*}
Thus the Crofton formula gives
\begin{align*}
\sum_{F\in{\mathcal F}_{d-1}(P)}V_{d-1}(F)&(R_{d-j,1}f_j^{v_F})(\langle v_F\rangle)
=c_{10}\sum_{F\in{\mathcal F}_{d-1}(P)}V_{d-1}(F)(I_jI_1f)(-v_F).
\end{align*}
If $K_j^d:C^\infty(S^{d-1})\to C^\infty(S^{d-1})$ is defined by
$$
(K_j^df)(v)=(R_{d-j,1}f_j^{v})(\langle v\rangle )\qquad\text{for }v\in S^{d-1},
$$
it can be seen that $K_j^d$ intertwines the group action of $SO(d)$ in the sense that
$(K_j^df)_\rho=K_j^df_\rho$ for each $\rho\in SO(d)$. It then follows from Schur's lemma that $K_j^d$ has the spaces of spherical harmonics as its eigenspaces. In particular, if $f\in C^\infty(S^{d-1})$ is centered then so is $K_j^df$, since scalar products are first degree spherical harmonics. The same observation is true of $I_jI_1f$. In this notation, our result above shows that
$$
\int_{S^{d-1}}(K_j^df)(v)\,S_{d-1}(K,dv)=c_{11}\int_{S^{d-1}}(I_jI_1f)(-v)\, S_{d-1}(K,dv)
$$
for all $f\in C^\infty(S^{d-1})$ and all polytopes $K$. A continuity argument implies that this is true for arbitrary convex bodies $K$. Thus $K_j^df$ and $I^*I_jI_1f$ differ only by a linear function. In particular, for centered $f$, they are the same. 

Hence, we obtain the following analogue to Koldobsky's orthogonality result, for arbitrary (centered) functions $f\in C^\infty(S^{d-1}) $.

\begin{theorem} For $1\le j\le d-1$, any centered function $f\in C^\infty(S^{d-1}) $ and all $v\in S^{d-1}$, we have
\begin{equation}\label{asymm}
(R_{d-j,1}(I_{d-j}f)_j^v)(\langle v\rangle)=b(d,j)\,(I_1f)(v)
\end{equation}
with
$$
b(d,j)=\frac{j\,2^{d-j-3}}{\pi(d-1)}\, 
\Gamma\left(\tfrac{d-j}{2}\right)^2 .
$$
\end{theorem}

The explicit value of the constant $b(d,j)$ is determined by letting $f$ be constant and by using the spherical 
integration formula
$$
\int_{S^{d-1}\cap N\cap v^\perp}\langle u,v\rangle^p\, \sigma_N(du)=\frac{1}{2\sqrt{\pi}}
\frac{\Gamma\left(\frac{d-j}{2}\right)\Gamma\left(\frac{p+1}{2}\right)}{\Gamma\left(\frac{d-j+p}{2}\right)}
$$
for $p>-1$, a unit vector $v\in N\in G(d,d-j)$ and $1\le j\le d-1$.

For comparison with the symmetric case, we note that, for even $f$, 
$$
(H_0^N f)(v)=\frac{1}{2}(R_{1,d-j}f)(N)
$$
is constant for unit vectors $v\in N\in G(d,d-j)$. Thus, for even $f$,
\begin{align*}
f_j^v(N)&=\frac{1}{2}(R_{1,d-j}f)(N)(H_{j+1}^N 1)(N)\\
&=  \frac{j}{8\sqrt{\pi}}\frac{\Gamma\left(\frac{d-j}{2}\right)\Gamma\left(\frac{j}{2}\right)}{
\Gamma\left(\frac{d+1}{2}\right)}\,(R_{1,d-j}f)(N),
\end{align*}
and so equation \eqref{asymm} is just (\ref{Grassorthog2}). For the extreme case, $j=d-1$, and arbitrary centered functions, \eqref{asymm} is trivial. For the other extreme case, $j=1$, we have
$$
(I_{d-1}f)_1^v(u^\perp)=(H_2^{u^\perp}H_0^{u^\perp}I_{d-1}f)(v) 
$$
for $u,v\in S^{d-1}$ with $\langle u,v\rangle=0$, 
and so the above result is, for any centered $f\in C^\infty(S^{d-1})$ and $v\in S^{d-1}$,
\begin{align*}
(I_1f)(v)
&=b(d,1)^{-1}\int_{S^{d-1}\cap v^\perp}\int_{S^{d-1}\cap u^\perp\cap v^+}\langle v,w\rangle^2\\
&\qquad \times \int_{S^{d-1}\cap u^\perp\cap w^+}(I_{d-1}f)(\theta)\,\sigma_{u^\perp}(d\theta)\,\sigma_{u^\perp}(dw)\,\sigma_{v^\perp}(du),
\end{align*}
which, unlike the corresponding case in Koldobsky's original result,  seems more complicated. 

If we restrict \eqref{asymm} to odd functions $f$, the result can be written, solely in terms of Fourier integral operators. To see this, let $I_p^{u^\perp}$ denote the Fourier integral operator calculated in the space $u^\perp$. Then, using  \cite[(2.15)]{GYY} twice, we get
\begin{align*}
(I_1f)(v)
&=\left(-\frac{1}{i\pi}\right)\frac{1}{\omega_{d-1}}b(d,1)^{-1}\int_{S^{d-1}\cap v^\perp}\int_{S^{d-1}\cap u^\perp\cap v^+}\langle v,w\rangle^2\\
&\qquad \times (I_0^{u^\perp}I_{d-1}f)(w)\sigma_{u^\perp}(dw)\,\sigma_{v^\perp}(du)\\
&=\left(-\frac{1}{i\pi}\right)\left(\frac{2}{i\pi}\right)\frac{1}{\omega_{d-1}^2}b(d,1)^{-1}
\int_{S^{d-1}\cap v^{\perp}}(I_{-2}^{u^\perp}I_0^{u^\perp}I_{d-1}f)(v)\,\sigma_{v^\perp}(du),
\end{align*}
and hence
$$
(I_1f)(v)=\frac{d-1}{2^{d-3}\pi^d}\int_{S^{d-1}\cap v^{\perp}}(I_{-2}^{u^\perp}I_0^{u^\perp}I_{d-1}f)(v)\,\sigma_{v^\perp}(du).
$$

It should be mentioned that formula \eqref{asymm}, once established, can be proved in a more direct way since both sides comprise intertwining operators. Since the multipliers (with respect to spherical harmonics) of these operators can be calculated, a comparison of the resulting values would be sufficient. The main task, which was completed by the proof we gave, was to find the formula, and to see it as an analogue of \eqref{Grassorthog}.

\noindent
Author's addresses:

\bigskip

\noindent
Paul Goodey, University of Oklahoma, Department of Mathematics, Norman, OK 73019, U.S.A., email: pgoodey@ou.edu

\bigskip
\noindent
Daniel Hug, Karlsruhe Institute of Technology (KIT), 
Department of Mathematics, 
D-76128 Karlsruhe, Germany, email: daniel.hug@kit.edu

\bigskip

\noindent
Wolfgang Weil, Karlsruhe Institute of Technology (KIT), 
Department of Mathe\-matics, 
D-76128 Karls\-ruhe, Germany, email: wolfgang.weil@kit.edu

\end{document}